\documentclass[11pt]{article}
\usepackage[utf8]{inputenc}
\usepackage[T1]{fontenc}
\usepackage[english]{babel}
\usepackage{amsmath,amssymb,amsthm,mathtools,mathrsfs,enumitem,fullpage,hyperref}
\usepackage{xcolor}
\usepackage{csquotes}
\usepackage{bbm}
\usepackage[backend=biber, style=numeric, sorting=nty]{biblatex}
\usepackage{authblk}

\numberwithin{equation}{section}

\newtheorem{theorem}{Theorem}[section]
\newtheorem{proposition}[theorem]{Proposition}
\newtheorem{lemma}[theorem]{Lemma}
\newtheorem{corollary}[theorem]{Corollary}

\theoremstyle{remark}
\newtheorem{remark}[theorem]{Remark}

\newtheorem{example}[theorem]{Example}

\theoremstyle{definition}
\newtheorem{definition}[theorem]{Definition}

\begin{document}

\title{Duality for a Martingale Transport Problem with Moment Constraints}

\author[1,3]{A.\,V. Kolesnikov}
\author[2,3]{A.\,V. Novikova}
\author[2,3]{K.\,O. Sokolov}

\affil[1]{HSE University}
\affil[2]{Lomonosov Moscow State University}
\affil[3]{Vega Institute Foundation}

\date{}

\maketitle

\begin{abstract}
We consider a weak martingale optimal transport problem related to the martingale analogue of the Benamou--Brenier formula. In contrast to the classical setting, the second marginal is not given, but specified only through a finite number of moment constraints of a particular form. For this problem, we derive a finite-dimensional dual formulation and prove the absence of a duality gap under the specified interiority condition on the constraint vector. In addition, we describe the structure of the optimal martingale coupling and its relation to Bass martingales.
\end{abstract}

\section{Introduction}

The optimal transport problem dates back to the work of G.~Monge~\cite{Monge1781}, where the problem of transporting mass at minimal cost was first formulated. The modern relaxed formulation of this problem was introduced by L.~V.~Kantorovich~\cite{Kantorovich1942,Kantorovich1958} and laid the foundations of optimal transport theory. By now, this theory has become one of the central tools of variational analysis, probability theory, and mathematical physics; see \cite{Villani2009}. In the case of quadratic cost, a notably important role is played by the dynamic Benamou--Brenier formula~\cite{BenamouBrenier2000}, which relates the $W_2$ distance to the minimization of kinetic energy over the curves of probability measures. Together with McCann's interpolation~\cite{McCann1997}, this formula provides one of the fundamental dynamic interpretations of optimal transport.

In recent years, martingale analogues of optimal transport problems have been studied extensively. The martingale constraint arises naturally, in particular, in model-independent pricing and hedging of financial derivatives; see \cite{BeiglbockHenryLaborderePenkner2013}. The corresponding mathematical theory of martingale optimal transport (MOT) was subsequently developed in~\cite{BeiglbockJuillet2016,BeiglbockNutzTouzi2017}. Martingale analogues of the Benamou--Brenier formulation (MBB) were introduced and studied in~\cite{BackhoffVeraguasBeiglbockHuesmannKallblad2020,HuesmannTrevisan2019}.

Let us briefly recall MBB formulation. Given a Brownian motion $(B_t)_{t \geq 0}$ and a predictable process $(\sigma_t)_{t \geq 0}$, define the martingale
\[
M_t:=M_0+\int_0^t\sigma_s\,dB_s,
\qquad
M_0\sim\mu,
\qquad
M_1\sim\nu .
\]
For fixed marginals $\mu$ and $\nu$, one considers the minimization problem
\[
\min_{\sigma_t, \,t \in [0, 1]} 
\mathbb{E}\int_0^1(1-\sigma_s)^2\,ds.
\]
It was shown in~\cite{BackhoffVeraguasBeiglbockHuesmannKallblad2020} that this dynamic problem admits a static formulation in terms of weak martingale optimal transport. Suppose $\gamma$ to be the standard Gaussian distribution on $\mathbb{R}$ and $\Pi_{\mathcal{M}}(\mu, \nu) = \{\pi \in \Pi(\mu, \nu): \; \int y \, \pi^x(dy) = x \ \mu \text{-a.s.}\}$ to be a set of martingale couplings between $\mu$ and $\nu$, then the corresponding static problem takes the form
\begin{equation}\label{eq:wmot}
\min_{\pi \in \Pi_\mathcal M(\mu, \nu)} 
\int W_2^2(\pi^x, \gamma)\,\mu(dx),
\end{equation}
which, up to an additive constant, is equivalent to
\[
\max_{\pi \in \Pi_\mathcal M(\mu, \nu)} 
\int \operatorname{MC}(\pi^x,\gamma)\,\mu(dx),
\]
where $\operatorname{MC}(\rho,\gamma):= \sup_{q\in\Pi(\rho,\gamma)}\int yz\,q(dy,dz)$. Further development of the duality theory for this problem and its connection with Bass martingales~\cite{Bass1983} was obtained in~\cite{BVBST,Tschiderer}.

In this paper, we consider a modification of problem~\eqref{eq:wmot} in which reference measure $\gamma$ is an absolutely continuous probability measure with a finite second moment (see Section~\ref{sec:problem_statement}), while the second marginal $\nu$ is no longer fully specified. Instead, only finitely many moments of a particular form are prescribed, for example, call option prices:
\[
\int (y-K_i)_+\,\nu(dy)=c_i,
\qquad
i=1,\ldots,N,
\]
where $K_1<\ldots<K_N$ are given strike prices. In terms of a martingale coupling, it means that we consider measures on $\mathbb{R}^2$
\[
\pi(dx,dy)=\pi^x(dy)\mu(dx),
\]
satisfying the martingale condition
\[
\int y\,\pi^x(dy)=x
\qquad
\mu\text{-a.s.}
\]
and a finite collection of constraints on the ``call moments'' $\overline C(\pi)=c \in \mathbb R^N$, where
$$
 \overline C(\pi):=
\left(
\int (y-K_1)_+\,\pi(dx,dy),
\ldots,
\int (y-K_N)_+\,\pi(dx,dy)
\right).
$$
This setup is natural from the viewpoint of applications: in financial markets, the terminal distribution of an asset price is typically unknown, whereas only a finite collection of option prices is observed. Thus, requiring the entire second marginal $\nu$ to be fixed is generally excessive. It is more natural to construct a martingale model calibrated only to the available option prices, especially in illiquid markets where only a small number of quotes are observed.

Let $\mathcal M_2(\mu)$ denote the class of square-integrable martingale couplings with first marginal $\mu$. For a given vector $c\in\mathbb R^N$, let $V(c)$ denote the optimal value of the following problem:
\begin{equation}\label{eq:main-problem}
    \sup_{\substack{\pi\in\mathcal M_2(\mu),\\ \overline C(\pi)=c}} 
    \left \{\int \operatorname{MC}(\pi^x,\gamma)\,\mu(dx) -\frac12\int y^2\,\pi(dx,dy) \right \}.
\end{equation}

Problems in which the terminal distribution is only partially specified have also been considered previously. In~\cite{GuoLoeperWang2022}, a finite collection of European option prices is imposed as a set of constraints in a semimartingale optimal transport problem with an integral functional of the local process  characteristics. In~\cite{AcciaioGarciaFloresMariniPammer2026}, moment constraints are introduced for the classical problem of minimizing the Wasserstein distance in $L^2$. In contrast to these settings, the subject of the present paper is the weak martingale optimal transport problem~\eqref{eq:main-problem} with a finite collection of constraints on the second marginal.

The main result of the paper is that problem~\eqref{eq:main-problem} with finitely many moment constraints admits a finite-dimensional dual formulation. More precisely, for every $\lambda\in\mathbb R^N$ we construct a convex function $\Psi_\lambda$ such that, for every ``admissible'' $c$ as defined in Section~\ref{sec:problem_statement},
$$
V(c)\leq
\inf_{\lambda\in\mathbb R^N}
\left\{
\lambda\cdot c-\int \Psi_\lambda^*(x)\,\mu(dx)
\right\},
$$
where $\Psi_\lambda^*$ denotes the Legendre transform of $\Psi_\lambda$. If, moreover, $c$ is an interior point of the set of admissible moment vectors, then there is no duality gap:
\begin{equation}\label{eq:intro-strong-dual}
V(c)=
\inf_{\lambda\in\mathbb R^N}
\left\{
\lambda\cdot c-\int \Psi_\lambda^*(x)\,\mu(dx)
\right\}.
\end{equation}

It is worth emphasizing the distinction between this result and the known duality theory for weak martingale optimal transport with a prescribed second marginal~\cite{BVBST}. In the classical setting, the dual variable is a function of the terminal state and corresponds to the constraint specified by the entire measure $\nu$. Here, by contrast, the constraints are finite-dimensional, and the dual problem therefore reduces to minimization over a vector $\lambda\in\mathbb{R}^N$. Moreover, due to the incomplete information about the second marginal, the main formula~\eqref{eq:intro-strong-dual} does not follow from the classical martingale Benamou--Brenier theorem: the optimal second marginal arises as part of the solution to the variational problem.

In addition to duality itself, we establish several further results. First, we provide a criterion for determining whether a given vector of call option prices $c$ satisfies a condition guaranteeing the absence of a duality gap. Second, we describe the structure of the optimal coupling for \eqref{eq:main-problem}: its conditional transition kernels have a piecewise-shift structure and may have atoms only at the strike points $K_1,\ldots,K_N$; see Corollary~\ref{cor:plan-structure}. Finally, we show that the optimal coupling admits a representation in terms of a Bass martingale and that the corresponding pair of marginals satisfies the irreducibility condition. Thus, the finite-dimensional calibration problem not only preserves the structure known for the MBB problem with full information about the second marginal, but also provides an alternative mechanism by which this structure arises: the terminal distribution is not prescribed in advance but instead emerges as part of the solution to the variational problem.

The paper is organized as follows. In Section~\ref{sec:problem_statement}, we explain how the static formulation~\eqref{eq:main-problem} arises from a weak martingale optimal transport problem with quadratic cost. We also introduce the auxiliary functions contained in the dual formula~\eqref{eq:intro-strong-dual}, define the set of admissible moment vectors, and establish a criterion for a vector $c \in \mathbb R^N$ to belong to its interior. In Section~\ref{sec:dual-theorem}, we prove the finite-dimensional duality theorem and show there is no duality gap under the interiority condition. As a corollary, we describe the structure of the optimal coupling reconstructed from a solution of the dual problem. Finally, in Section~\ref{sec:mbb}, we establish the connection between the resulting solution, Bass martingales, and the MBB problem with a given second marginal under the original assumption $\gamma = \mathcal N(0,1)$.

\section{Problem statement and preliminaries}\label{sec:problem_statement}

Let $\mu \in \mathcal P_2(\mathbb R)$, and let $\gamma \in \mathcal P_2(\mathbb R)$ be absolutely continuous  with respect to Lebesgue measure. Recall that $\mathcal M_2(\mu)$ denotes the class of all probability measures $\pi(dx,dy)=\pi^x(dy)\mu(dx)$ on $\mathbb R^2$ satisfying
\[
\int y\,\pi^x(dy)=x\quad\mu\text{-a.s.}, \qquad \int y^2\,\pi(dx,dy)<\infty.
\]
Throughout the paper we use the existence of regular conditional distributions and measurable selection theorems in the form given in \cite[Chapter~10]{Bogachev2007}.

Let $\pi \in \mathcal M_2(\mu)$ and let $K_1<\dots<K_N \in \mathbb{R}$ be strike prices of european call-options. The components of the vector
\[
\overline C_K(\pi) = \overline C(\pi) = \left( \int (y - K_1)_+\,\pi(dx,dy),\dots, \int (y - K_N)_+\,\pi(dx,dy) \right)\in\mathbb R^N
\]
will be referred to as \textit{call moments} of the coupling $\pi$. Since the strike vector $K = (K_1, \ldots, K_N)$ is usually fixed throughout, we mostly omit the subscript. A set of \textit{admissible} call-moment vectors is defined by
\[
\mathcal C:= \{\overline C(\pi):\pi \in \mathcal M_2(\mu)\} \subset\mathbb R^N.
\]
Thus, if $c\in\mathcal C$, then there exists a coupling $\pi\in\mathcal M_2(\mu)$ such that $c=\overline C(\pi)$.

\begin{remark}\label{rem:alternative-statement}
Another natural formulation, in place of \eqref{eq:main-problem}, is to consider
\begin{equation} \label{eq:main-alternative} 
\sup_{\substack{\pi\in\mathcal M_2(\mu),\\ \overline C(\pi)=c}} \left \{\int \operatorname{MC}(\pi^x,\gamma)\,\mu(dx)\right \}.
\end{equation}
In this formulation, the dual problem has a less complicated form. More precisely, one can show that the value of \eqref{eq:main-alternative} coincides with
\[
\inf_{\substack{\lambda\in\mathbb R^N,\\ \lambda_i\geq 0,\ i=1,\ldots,N}} \left \{\lambda \cdot c - \int \left (g^* * \gamma \right )^*(x)\, \mu(dx) \right \}
\]  
where 
\[
g_\lambda(y)=\sum_{i=1}^N\lambda_i(y-K_i)_+, \qquad (\varphi * \gamma) (t) = \int \varphi(t+z)\, \gamma(dz).
\]
Although this formulation admits a simpler description (in particular, explicit formulas can be obtained, unlike for problem \eqref{eq:main-problem}), it has a substantial drawback: a solution to \eqref{eq:main-alternative} may not exist. One can show that, for $N=1$ and $\gamma$ being uniform on $[0,1]$, the primal problem has no optimizer. The conditional measures along which convergence to the optimum is observed are of the form
\[
\pi^x_n = \frac{n-x}{n-K} \delta_K + \frac{x-K}{n-K} \delta_n, \quad \text{if } x \geq K,
\]
and
\[
\pi^x_n = \frac{n+x}{n+K} \delta_K + \frac{K-x}{n+K} \delta_{-n}, \quad \text{if } x < K.
\]
\end{remark}

To streamline formulas, for a probability measure $\rho \in \mathcal P_2(\mathbb R)$, denote its call price function by
\[
C_\rho(K):=\int (x-K)_+\,\rho(dx),
\qquad K\in\mathbb R.
\]
The following result provides a simple description of the interior of the set $\mathcal C$, which will be needed further to verify the condition of the duality theorem.

\begin{proposition}\label{prop:c-check}
Let $c=(c_1,\dots,c_N) \in \mathbb R^N$. Then $c\in\operatorname{int}\mathcal C$
if and only if the following conditions hold:
\begin{enumerate}
    \item $c_i>C_\mu(K_i)$ for all $i=1,\dots,N$;
    \item if $N\geq2$, then
    \[
    -1<s_1<\cdots<s_{N-1}<0,
    \qquad \text{where }\;
    s_i:=\frac{c_{i+1}-c_i}{K_{i+1}-K_i},
    \quad i=1,\dots,N-1.
    \]
\end{enumerate}
\end{proposition}

\noindent
The conditions of the proposition provide an explicit criterion for a given call-moment vector to be admissible and to belong to $\operatorname{int} \mathcal{C}$.

\begin{proof}
We first prove necessity. Let $c\in\operatorname{int}\mathcal C$. Then $c=\overline C(\pi)$ for some $\pi\in\mathcal M_2(\mu)$. Let $\nu$ denote the second marginal of $\pi$. The martingale condition implies $\mu\preceq_{\mathrm{cx}}\nu$. Hence, $c_i = C_\nu(K_i)\geq C_\mu(K_i)$. Moreover, $C_\nu$ is convex, nonincreasing, and $1$-Lipschitz. Therefore, if $N\geq2$,
\[
-1\le s_1\le\cdots\le s_{N-1}\le0.
\]

If equality holds in any of the inequalities above, then $c$ lies on the boundary of the corresponding closed half-space containing $\mathcal C$, which contradicts the assumption that $c$ is an interior point.

We now prove sufficiency. Denote $\overline\mu:=\int_{\mathbb R}x\,\mu(dx)$. We construct a call price function $C$ taking the prescribed values $c_1,\dots,c_N$ at $K_1,\dots,K_N$ and corresponding to some measure $\nu\in\mathcal P_2(\mathbb R)$ such that $\mu\preceq_{\mathrm{cx}}\nu$. Suppose first that $N\geq2$: consider a piecewise-linear function passing through the points $\{(K_i,c_i)\}_{i=1}^N$. Its slope on each interval $(K_i,K_{i+1})$ is $s_i$, and it remains to extend it to $(-\infty, K_1)$ and $(K_N, \infty)$. By Jensen's inequality,
\[
c_1>C_\mu(K_1)\geq (\overline \mu-K_1)_+\geq \overline \mu-K_1,
\]
so that $c_1-(\overline \mu-K_1) >0$. Choose $K_0<K_1$ such that the slope
\[
s_0:=\frac{c_1-(\overline \mu-K_0)}{K_1-K_0} = -1+\frac{c_1-(\overline \mu-K_1)}{K_1-K_0}
\]
satisfies $-1<s_0<s_1$. Also, since $c_N > C_\mu(K_N) \geq 0$, we have $c_N>0$.
Choose $K_{N+1}>K_N$ such that
\[
s_N:=\frac{0-c_N}{K_{N+1}-K_N} = -\frac{c_N}{K_{N+1}-K_N}
\]
satisfies $s_{N-1}< s_N < 0$.

Define $L:\mathbb R\to\mathbb R$ as a piecewise-linear function through the points $ (K_0,\overline \mu-K_0)$, $(K_1,c_1),\ldots,(K_N,c_N)$,  $(K_{N+1},0)$, and set additionally
\[
L(K):=\overline \mu-K,\qquad K < K_0,
\]
and
\[
L(K):=0,\qquad K > K_{N+1}.
\]
Then $L$ is a convex, nonincreasing, piecewise-linear function that satisfies $L(K_i)=c_i$ for $i=1,\dots,N$. Its slopes belong to $[-1,0]$, and hence it is $1$-Lipschitz. Now, set
\[
C(K):=\max\{L(K),C_\mu(K)\},\qquad K\in\mathbb R.
\]
The function $C$ retains all the properties listed above for $L$, since $C_\mu$ is also convex, non-increasing, and $1$-Lipschitz by definition, and $C_\mu(K_i) < L(K_i)$.

To establish the existence of $\nu$ corresponding to $C(K) = \int (y - K)_+ \,\nu(dy)$, it remains to verify the limiting behaviour of $C$. For all sufficiently large $K$, we have $L(K)=0$, and therefore
\[
C(K)=\max\{0, C_\mu(K)\}= C_\mu(K),
\]
so that $C(K)\to 0$ as $K\to+\infty$. Similarly, for all sufficiently small $K$, we have $L(K) = \overline\mu -K$. Moreover,
\[
C_\mu(K)=\int (x-K)_+\,\mu(dx) \geq \int (x-K)\,\mu(dx) = \overline\mu -K.
\]
Hence, $C(K) = C_\mu(K)$ for all sufficiently small $K$, and therefore $C(K) + K \to \overline\mu$ as $K\to-\infty$.

Consequently, $C$ is the call price function associated with some probability measure $\nu$ with mean $\overline\mu$:
\[
C(K)=\int (y-K)_+\,\nu(dy),
\]
since the survival function $F_\nu(K):= \nu((K, \infty)) := -C'_+(K)$ is well defined for every $K \in \mathbb R$.

Let us prove $\nu \in \mathcal P_2(\mathbb R)$. Splitting the second moment over $\mathbb R_+$ and $\mathbb R_-$ yields
\[
\int_{\mathbb R_+} y^2\, \nu(dy) = \int_{\mathbb R}(y_+)^2\,\nu(dy) = 2 \int_0^\infty C(K)\,dK,
\]
\[
\int_{\mathbb R_-} y^2\, \nu(dy) = \int_{\mathbb R}(y_-)^2\,\nu(dy) = 2 \int_{-\infty}^0 P_\nu(K)\,dK,
\]
where $P_\nu(K):=\int (K-y)_+\,\nu(dy)$. We have already shown that $C(K) = C_\mu(K)$ for all sufficiently large $K$. As $\mu \in \mathcal P_2(\mathbb R)$,
\[
\int_0^\infty C(K)\,dK<\infty.
\]
Since $\nu$ has mean $\overline \mu$ and $C(K) = C_\mu(K)$ for all sufficiently small $K$,
\[
P_\nu(K) = C(K) + K - \overline\mu = C_\mu(K) + K - \overline\mu = P_\mu(K).
\]
Consequently,
\[
\int_{-\infty}^0 P_\nu(K)\,dK<\infty.
\]
It follows that $\int_{\mathbb R}y^2\,\nu(dy)<\infty$, that is, $\nu\in\mathcal P_2(\mathbb R)$.

Thus, we obtain measures $\mu$ and $\nu$ in $\mathcal P_2(\mathbb R)$ such that $\mu \preceq_{\mathrm{cx}} \nu$ by construction of $C$. By Strassen's theorem \cite[Theorem 9]{Strassen1965}, there exists a martingale coupling $\pi$ with first marginal $\mu$ and second marginal $\nu$. Moreover,
\[
\int_{\mathbb R^2} y^2\,\pi(dx,dy) = \int_{\mathbb R} y^2\,\nu(dy)<\infty,
\]
so that $\pi\in\mathcal M_2(\mu)$. In addition,
\[
\int (y - K_i)_+\, \pi(dx, dy) = \int (y - K_i)_+\, \nu(dy) = C(K_i) = c_i,
\]
hence $c \in \mathcal C$.

We now show that $c$ has a neighbourhood entirely contained in $\mathcal C$. Since all inequalities in the statement are strict and the slopes $s_i$ depend continuously on $c$, there exists $\varepsilon>0$ such that every $c'=(c'_1,\ldots,c'_N) \in B_\varepsilon(c)$ satisfies
\[
c'_i>C_\mu(K_i),\qquad i=1,\dots,N,
\]
and, if $N\geq2$,
\[
s'_i:=\frac{c'_{i+1}-c'_i}{K_{i+1}-K_i} \in (-1, 0), \qquad i=1,\dots,N-1,
\]
and
\[
s'_1 < s'_2 < \ldots < s'_{N-1}.
\]
Thus, $c'$ satisfies the conditions in the statement, hence, as already proved, $c'\in\mathcal C$. Therefore, $B_\varepsilon(c)\subset \mathcal C$, and consequently $c\in\operatorname{int}(\mathcal C)$.

The case $N=1$ is treated similarly: the piecewise-linear function is constructed through the single point $(K_1, c_1)$.
\end{proof}

For a probability measure $\rho\in\mathcal P_2(\mathbb R)$, consider
\[
\operatorname{MC}(\rho,\gamma)=
\sup_{q\in\Pi(\rho,\gamma)}\int yz\,q(dy,dz).
\]
For $\pi\in\mathcal M_2(\mu)$, define
\[
J_\gamma(\pi):=
\int \operatorname{MC}(\pi^x,\gamma)\,\mu(dx)
-\frac12\int y^2\,\pi(dx,dy).
\]
Without loss of generality, throughout what follows we assume that
\[
\int z\, \gamma(dz) = 0, \qquad\int z^2\, \gamma(dz) = 1.
\]
Indeed, denote
\[
m_\gamma:=\int z\,\gamma(dz), \qquad m_\mu:=\int  x\,\mu(dx)
\]
and $\gamma_0:=(z\mapsto z-m_\gamma)_\#\gamma$. Then for every $\rho\in\mathcal P_2(\mathbb R)$
\[
\operatorname{MC}(\rho,\gamma) = \operatorname{MC}(\rho,\gamma_0) + m_\gamma\int y\,\rho(dy).
\]
Consequently, for every $\pi\in\mathcal M_2(\mu)$, the martingale condition gives
\[
J_\gamma(\pi) = J_{\gamma_0}(\pi) + m_\gamma\, m_\mu.
\]
Thus, the two optimization problems have the same optimal couplings, and their objective functionals differ by an additive constant independent of $\pi$.

Suppose now that $\gamma$ is centred and let $\sigma^2:=\int z^2\,\gamma(dz)>0$. For the scaling map $S_\sigma(x):=x/\sigma$, set
\[
\hat\gamma: (S_\sigma)_\#\gamma,\qquad
\hat\mu:=(S_\sigma)_\#\mu,\qquad
\hat\pi:=(S_\sigma, S_\sigma)_\#\pi,
\]
and also
\[
\hat K_i:=\frac{K_i}{\sigma}, \qquad \hat c_i:=\frac{c_i}{\sigma}.
\]
Then $\hat\gamma$ has unit second moment, and the first marginal of $\hat\pi$ is $\hat\mu$. Moreover, by the definition of $\hat\pi$, for $\hat\mu$-a.e. $\hat x$, with $x = \sigma \hat x$,
\[
\hat\pi^{\hat x}=(S_\sigma)_\#\pi^x.
\]
In particular,
\[
\int \hat y\,\hat\pi^{\hat x}(d\hat y) = \frac1\sigma \int y\,\pi^x(dy) = \frac{x}{\sigma} = \hat x, \qquad
\int \hat y^2\,\hat\pi(d\hat x,d\hat y) = \frac1{\sigma^2}\int y^2\,\pi(dx,dy)<\infty
\]
and therefore $\hat\pi \in \mathcal M_2(\hat\mu)$. Furthermore, the map $q\longmapsto (S_\sigma,S_\sigma)_\#q$ defines a one-to-one correspondence between $\Pi(\pi^x,\gamma)$ and $\Pi(\hat\pi^{\hat x},\hat\gamma)$. Therefore,
\[
\operatorname{MC}(\pi^x,\gamma) = \sigma^2
\operatorname{MC}(\hat\pi^{\hat x},\hat\gamma),
\]
and hence $J_\gamma(\pi) = \sigma^2 J_{\hat\gamma}(\hat\pi)$. Moreover,
\[
\int (\hat y - \hat K_i)_+\, d\hat\pi = \frac{1}{\sigma}\int (y - K_i)_+\, d\pi, \quad i = 1, \ldots, N.
\]
Thus, the condition $\overline C_K(\pi)=c$ is equivalent to $\overline C_{\hat K}(\hat\pi)=\hat c$, and the scaling map establishes a one-to-one correspondence between admissible couplings of the original and normalized problems. Consequently, problem \eqref{eq:main-problem} reduces to an equivalent problem with a centered reference measure of unit second moment. 

In what follows, we suppress the subscript $\gamma$ and write $J$ for $J_\gamma$. We now establish the relation between the functional $J$ and the $W_2$ distance.

\begin{lemma}\label{lem:square}
For every $\rho\in\mathcal P_2(\mathbb R)$,
\begin{equation}\label{eq:square}
\operatorname{MC}(\rho,\gamma) - \frac12 \int y^2\, \rho(dy) = \frac12 - \frac12 W_2^2 (\rho,\gamma).
\end{equation}
In particular, for $\pi \in \mathcal M_2(\mu)$,
\[
-\infty < J(\pi) = \frac12 - \frac12 \int  W_2^2 (\pi^x,\gamma)\, \mu(dx) \leq  \frac12.
\]
\end{lemma}

\begin{proof}
For every coupling $q\in\Pi(\rho,\gamma)$, we have
\[
\int yz\,q(dy,dz)-\frac12\int y^2\,\rho(dy)=\frac12-\frac12\int (y-z)^2\,q(dy,dz).
\]
Therefore, taking the supremum over $q\in\Pi(\rho,\gamma)$ is equivalent to taking the infimum of the integral of the quadratic cost, which equals $W_2^2(\rho,\gamma)$ by definition. This proves \eqref{eq:square}.

To obtain the bounds on $J(\pi)$, it suffices to use the nonnegativity of $W^2_2(\pi^x,\gamma)$ together with
\[
\int (y-z)^2\, q(dy, dz) \leq 2\int y^2\, \pi^x(dy) + 2 < \infty \qquad \forall q \in \Pi(\pi^x, \gamma), 
\]
since $\pi \in \mathcal M_2(\mu)$.
\end{proof}

Finally, for $c\in\mathcal C$, we can write \eqref{eq:main-problem} in the simplified form
\[
V(c)=
\sup\{J(\pi):\pi\in\mathcal M_2(\mu),\ \overline C(\pi)=c\}.
\]
Lemma \ref{lem:square} shows how the weak martingale optimal transport problem \eqref{eq:wmot} leads to problem \eqref{eq:main-problem} with moment constraints.

We now introduce the notation needed to formulate the dual problem. For
$\lambda=(\lambda_1,\dots,\lambda_N)\in\mathbb R^N$, define
\[
\Lambda_\lambda(y):=\sum_{i=1}^N\lambda_i(y-K_i)_+.
\]
Set
\[
L_\lambda:=\sum_{i=1}^N|\lambda_i|,
\qquad
B_\lambda:=\sum_{i=1}^N|\lambda_i|\,|K_i|.
\]
Then, for every $y\in\mathbb R$,
\begin{equation}\label{eq:Lambda-bound}
|\Lambda_\lambda(y)|
\leq
\sum_{i=1}^N|\lambda_i|\,\bigl(|y|+|K_i|\bigr)
=
L_\lambda |y|+B_\lambda.    
\end{equation}
Next, define
\[
H_\lambda(u):=
\sup_{y\in\mathbb R}
\left\{
uy-\frac12y^2-\Lambda_\lambda(y)
\right\},
\qquad u\in\mathbb R,
\]
and establish the key properties of this function.

\begin{lemma}\label{lem:H-properties}
For every $\lambda\in\mathbb R^N$, the function $H_\lambda(u)$ has the following properties:
\begin{enumerate}[label=(\roman*)]
\item $H_\lambda$ is finite and convex, and for every $u\in\mathbb R$,
\begin{equation}\label{eq:H-growth}
\frac12u^2-L_\lambda |u|-B_\lambda
\leq H_\lambda(u)
\leq \frac12(|u|+L_\lambda)^2+B_\lambda.
\end{equation}

\item For each $u\in\mathbb R$, the set of maximizers
\[
\Gamma_\lambda(u):=\operatorname*{argmax}_{y\in\mathbb R}\Bigl\{uy-\frac12 y^2-\Lambda_\lambda(y)\Bigr\}
\]
is nonempty and compact.

\item For every $u\in\mathbb R$ and any $y\in \Gamma_\lambda(u)$,
\begin{equation}\label{eq:T-growth-maximizers}
|y-u|\leq L_\lambda.
\end{equation} 

\item We have
\[
\min \Gamma_\lambda(u)=H'_{\lambda,-}(u),\qquad \max \Gamma_\lambda(u)=H'_{\lambda,+}(u).
\]
In particular,
\begin{equation}\label{eq:T-growth}
|\min \Gamma_\lambda(u)|\leq L_\lambda + |u|.
\end{equation}
\end{enumerate}
\end{lemma}

\begin{proof}
\begin{enumerate}[label=(\roman*)]
\item By \eqref{eq:Lambda-bound},
\[
uy-\frac12 y^2-\Lambda_\lambda(y) \leq (|u|+L_\lambda)|y|-\frac12 y^2+B_\lambda \leq \frac12(|u|+L_\lambda)^2+B_\lambda.
\]
Taking the supremum over $y$ proves the upper bound in \eqref{eq:H-growth}. The lower bound follows by taking $y=u$ and using \eqref{eq:Lambda-bound} again:
\[
H_\lambda(u)\geq u^2-\frac12u^2-\Lambda_\lambda(u) \geq \frac12u^2-L_\lambda|u|-B_\lambda.
\]
Thus, $H_\lambda$ is finite. Its convexity comes from the fact that the supremum is taken over the family of affine functions $u \longmapsto uy - \frac{1}{2} y^2 - \Lambda_\lambda(y)$.

\item For fixed $u$, the function
\[
y\longmapsto uy-\frac12 y^2-\Lambda_\lambda(y)
\]
is continuous in $y$. Moreover,
\[
uy-\frac12 y^2-\Lambda_\lambda(y) \leq (|u|+L_\lambda)|y|-\frac12 y^2+B_\lambda \xrightarrow[|y|\to\infty]{} -\infty.
\]
Hence, the maximum is attained, and $\Gamma_\lambda(u)$ is nonempty and compact.

\item Let $u\in\mathbb R$ and $y\in\Gamma_\lambda(u)$. Then for every $h>0$
\[
H_\lambda(u) = uy-\frac12 y^2-\Lambda_\lambda(y) \geq u(y+h)-\frac12 (y+h)^2-\Lambda_\lambda(y+h).
\]
Rearranging gives
\[
0\geq h(u-y)-\frac12 h^2-\bigl(\Lambda_\lambda(y+h)-\Lambda_\lambda(y)\bigr) \geq h(u-y)-\frac12 h^2-L_\lambda h,
\]
since
\[
|\Lambda_\lambda(y+h)-\Lambda_\lambda(y)| \leq L_\lambda h.
\]

Dividing by $h>0$ and letting $h\downarrow0$, we obtain
\[
u-y\leq L_\lambda.
\]

Similarly, comparing the values of $y\longmapsto uy-\frac12 y^2-\Lambda_\lambda(y)$ in $y$ and $y-h$, we obtain $y-u\leq L_\lambda$.
Thus, \eqref{eq:T-growth-maximizers} holds.

\item
Since $H_\lambda(u)$ is convex, $\partial H_\lambda(u) = [H'_{\lambda, -}, H'_{\lambda, +}] = \text{conv}\left (\Gamma_\lambda(u) \right )$, that is, $\min \Gamma_\lambda(u) = H'_{\lambda, -}(u)$ and $\max \Gamma_\lambda(u) = H'_{\lambda, +}(u)$. Indeed, if $y\in\Gamma_\lambda(u)$, then for every $v\in\mathbb R$
\[ 
H_\lambda(v)\geq vy-\frac12y^2-\Lambda_\lambda(y) = H_\lambda(u)+y(v-u). 
\] 
Thus, $y\in\partial H_\lambda(u)$, and therefore
\[ 
H'_{\lambda,-}(u) \leq \min \Gamma_\lambda(u) \leq \max\Gamma_\lambda(u)\leq H'_{\lambda,+}(u). 
\] 

Let us prove the reverse inequalities. Suppose that $h>0$ and $y_h\in\Gamma_\lambda(u-h)$. Then
\[ 
H_\lambda(u)\geq uy_h-\frac12y_h^2-\Lambda_\lambda(y_h) =H_\lambda(u-h)+hy_h, 
\] 
that is,
\[ 
y_h\leq \frac{H_\lambda(u)-H_\lambda(u-h)}h\leq H'_{\lambda,-}(u). 
\] 
By part (iii), the family $\{y_h:0<h\le1\}$ is bounded. Subsequently, as $h\downarrow0$, we may choose a convergent subsequence $y_{h_n}\to y$. By the continuity of $H_\lambda$ and the function $(u,y)\mapsto uy-\frac12y^2-\Lambda_\lambda(y)$, we have
$y\in\Gamma_\lambda(u)$, and hence
\[
H'_{\lambda,-}(u)\geq y\geq \min\Gamma_\lambda(u).
\]
Similarly, $H'_{\lambda,+}(u)\leq \max \Gamma_\lambda(u)$. Therefore, $H'_{\lambda,-}(u) = \min\Gamma_\lambda(u)$ and $H'_{\lambda,+}(u) = \max\Gamma_\lambda(u)$.

Finally, \eqref{eq:T-growth} follows directly from \eqref{eq:T-growth-maximizers}.
\end{enumerate}
\end{proof}

By estimate \eqref{eq:H-growth}, the convolution of $H_\lambda$ with reference distribution $\gamma$ is well defined:
\[
\Psi_\lambda(a):= H_\lambda * \gamma = \int H_\lambda(a+z)\,\gamma(dz),\qquad a\in\mathbb R.
\]
We also denote $T_\lambda(u):= \min \Gamma_\lambda(u) = H'_{\lambda,-}(u)$. Since the left derivative of a convex function is Borel measurable and non-decreasing, the same properties hold for $T_\lambda$ as a function of $u$.

\begin{lemma}\label{lem:Psi-properties}
For every $\lambda \in \mathbb R^N$ the function $\Psi_\lambda(a)$ is finite, convex, and continuously differentiable on $\mathbb R$, with
\begin{equation}\label{eq:Psi-derivative}
\Psi_\lambda'(a)=\int T_\lambda(a+z)\,\gamma(dz),\qquad a\in\mathbb R.
\end{equation}
Moreover, the following estimate holds:
\begin{equation}\label{eq:Psi-prime-growth}
|\Psi_\lambda'(a)-a|\le L_\lambda, \qquad a\in\mathbb R.
\end{equation}
In addition, for every $x\in\mathbb R$, the function $a\longmapsto ax - \Psi_\lambda(a)$ attains its maximum at $\mathbb R$.
\end{lemma}

\begin{proof}
It follows from \eqref{eq:H-growth} that the function $z\mapsto H_\lambda(a+z)$ is integrable with respect to $\gamma$ for every $a\in\mathbb R$. Hence, $\Psi_\lambda$ is well defined and finite, and is convex by Lemma \ref{lem:H-properties}.

We now show that $\Psi_\lambda\in C^1(\mathbb R)$. Since $H_\lambda$ is convex, its one-sided derivatives exist everywhere and the set where they differ is, at most, countable. Moreover, according to Lemma~\ref{lem:H-properties},
\[ 
|H'_{\lambda,\pm}(u)|\leq |u|+L_\lambda. 
\] 
Fix $a\in\mathbb R$. By the convexity of $H_\lambda$, for $h>0$,
\[ 
H'_{\lambda,+}(a-h+z) \leq \frac{H_\lambda(a+z)-H_\lambda(a-h+z)}h \leq H'_{\lambda,-}(a+z). 
\] 
For $0<h\leq1$, this difference quotient is dominated by $|a+z|+1+L_\lambda$, which is integrable with respect to $\gamma$. Therefore, by Lebesgue's dominated convergence theorem
\[ 
\Psi'_{\lambda,-}(a)=\int H'_{\lambda,-}(a+z)\,\gamma(dz). 
\] 
Similarly,
\[ 
\Psi'_{\lambda,+}(a)=\int H'_{\lambda,+}(a+z)\,\gamma(dz). 
\] 
Since $\gamma$ is absolutely continuous and the set of non-differentiability points of $H_\lambda$ is at most countable, these two integrals coincide. Hence, $\Psi_\lambda$ is differentiable and
\[ 
\Psi_\lambda'(a)=\int H'_{\lambda,-}(a+z)\,\gamma(dz) = \int T_\lambda(a+z)\,\gamma(dz). 
\] 
The continuity of $\Psi'_\lambda$ follows from the monotonicity of $T_\lambda$ and \eqref{eq:T-growth}.

Next, using \eqref{eq:T-growth} once again, we obtain \eqref{eq:Psi-prime-growth}, since $\int z\,\gamma(dz)=0$:
\[
|\Psi_\lambda'(a)-a| = \left| \int \bigl( T_\lambda(a+z) - a+z) \bigr) \,\gamma(dz) \right| \leq \int \bigl| T_\lambda(a+z) - (a+z) \bigr| \,\gamma(dz) \leq L_\lambda.
\]

Finally, the lower bound in \eqref{eq:H-growth} gives
\[
H_\lambda(a+z)\geq \frac12(a+z)^2-L_\lambda|a+z|-B_\lambda.
\]
Integrating that with respect to $\gamma$, we obtain
\begin{equation}\label{eq:Psi-bounds}
\Psi_\lambda(a)\geq \frac12 a^2-C_\lambda|a|-C'_\lambda
\end{equation}
for some finite constants $C_\lambda,C'_\lambda$. Consequently,
\[
ax - \Psi_\lambda(a)\to -\infty \quad \text{as } |a|\to\infty
\]
for every $x\in\mathbb R$. Hence, the continuous function $a\mapsto ax-\Psi_\lambda(a)$ reaches its maximum in $\mathbb R$.
\end{proof}

By the preceding lemma, we may define $\Psi_\lambda^*$, the Legendre transform of $\Psi_\lambda$, by
\[
\Psi_\lambda^*(x):=\sup_{a\in\mathbb R}\{ax-\Psi_\lambda(a)\}.
\]

\section{Duality theorem}\label{sec:dual-theorem}

\begin{theorem}\label{thm:main}
For every $c\in\mathcal C$ we have
\begin{equation}\label{eq:weak-dual} 
V(c)\leq \inf_{\lambda\in\mathbb R^N}\left\{\lambda\cdot c-\int \Psi_\lambda^*(x)\,\mu(dx)\right\}.
\end{equation}
Assume, in addition, that $c \in \operatorname{int} (\mathcal{C})$. Then there is no duality gap:
\begin{equation}\label{eq:strong-dual}
V(c)= \inf_{\lambda\in\mathbb R^N}\left\{\lambda\cdot c-\int \Psi_\lambda^*(x)\,\mu(dx)\right\}.
\end{equation}
\end{theorem}

\noindent
Formula \eqref{eq:strong-dual} is an analogue of the duality result for weak martingale optimal transport in \cite{BVBST}; the main difference lies in the construction of the function $H_\lambda$, which reflects incomplete information about the second marginal. Note that the dual functional is well defined since $\Psi^*_\lambda \in L^1(\mu)$. Indeed, taking $a=0$ and using \eqref{eq:Psi-bounds}, we obtain the two-sided estimate
\[
-\Psi_\lambda(0) \leq \Psi^*_\lambda(x) \leq \frac12(|x| + C_\lambda)^2 + C'_\lambda.
\]

Before proving the theorem, we establish several important properties of $J(\pi)$ and $V(c)$.

\begin{lemma}\label{lem:optimal-pi-existence}
For every $c\in\mathcal C$, there exists $\pi\in\mathcal M_2(\mu)$ such that
\[
V(c)=J(\pi).
\]
\end{lemma}

\begin{proof}
Set $\mathcal A_c:=\{\pi\in\mathcal M_2(\mu):  \overline C(\pi)=c\}$. Since $c\in\mathcal C$, the set $\mathcal A_c$ is nonempty. By Lemma~\ref{lem:square}, maximizing $J$ over $\mathcal A_c$ is equivalent to minimizing functional $I(\pi):=\int W_2^2(\pi^x,\gamma)\,\mu(dx)$ over $\mathcal A_c$. Denote
\[
\alpha:=\inf_{\pi\in\mathcal A_c} I(\pi).
\]
Since $\mathcal A_c\ne\varnothing$ and $I(\pi)<\infty$ for any $\pi\in\mathcal M_2(\mu)$, we have $0\leq\alpha<\infty$.

Let $\{\pi_n\}\subset\mathcal A_c$ be a minimizing sequence: $I(\pi_n)\downarrow\alpha$. For each $\rho\in\mathcal P_2(\mathbb R)$
\[
\int y^2\,\rho(dy) \leq 2\,W_2^2(\rho,\gamma) + 2.
\]
Applying this estimate to $\rho=\pi_n^x$ for every $n$ and integrating with respect to $\mu$, we obtain
\[
\int y^2\,\pi_n(dx,dy) \leq 2\, I(\pi_n)+2 < \infty.
\]
The first marginal of $\pi_n$ is $\mu$, and hence
\begin{equation}\label{eq:square-finite}
\sup_n\int (x^2+y^2)\,\pi_n(dx,dy)<\infty.
\end{equation}
In particular, 
the family $\{\pi_n\}$ is uniformly tight. By Prokhorov's theorem, there exist a subsequence $\{\pi_{n_k}\}$, which, for simplicity, we continue to denote by $\{\pi_n\}$, and a probability measure $\pi$ such that
\[
\pi_n\Rightarrow\pi.
\]
Moreover, $\pi\in\mathcal P_2(\mathbb R^2)$ since
\[
\int (x^2 + y^2)\,\pi(dx,dy) \leq \liminf_{n\to\infty} \int (x^2 + y^2)\,\pi_n(dx,dy) <\infty.
\]

We now show that $\pi\in\mathcal A_c$. The first marginal of $\pi$ is $\mu$ since, for every $\varphi\in C_b(\mathbb R)$,
\[
\int \varphi(x)\,\pi(dx,dy) = \lim_{n\to\infty}\int \varphi(x)\,\pi_n(dx,dy) = \int \varphi(x)\,\mu(dx).
\]
Next, the uniform boundedness of the second moments in \eqref{eq:square-finite} implies 
uniform integrability with respect to $\{\pi_n\}$ of any function of at most linear growth. Consequently, for every $\varphi\in C_b(\mathbb R)$, we may pass to the limit in
\[
\int \varphi(x)(y-x)\,\pi(dx,dy) = \lim_{n \to \infty}\int \varphi(x)(y-x)\,\pi_n(dx,dy)=0.
\]
Since the first marginal of $\pi$ is $\mu$, this is equivalent to the martingale condition
\[
\int y\,\pi^x(dy)=x \qquad \mu\text{-a.s.}
\]
Similarly, since the functions $y \mapsto (y-K_i)_+$ are continuous and have linear growth,
\[
\int (y-K_i)_+\,\pi(dx,dy) = \lim_{n\to\infty}\int (y-K_i)_+\,\pi_n(dx,dy) = c_i, \qquad i=1,\dots,N.
\]
Therefore, $\pi\in\mathcal M_2(\mu)$ and $  \overline C(\pi)=c$, that is, $\pi\in\mathcal A_c$.

It remains to prove the lower semicontinuity of $I$. For a measure $\eta \in \mathcal P_2(\mathbb R^2)$ with first marginal $\mu$, set
\[
\mathcal Q(\eta):= \left\{ Q\in\mathcal P(\mathbb R^3): Q_{x,y}=\eta,\; Q_{x,z}=\mu\otimes\gamma
\right\}.
\]
Then
\[
I(\eta) = \inf_{Q\in\mathcal Q(\eta)} \int (y-z)^2\,Q(dx,dy,dz).
\]

\noindent
For each $n$, choose $Q_n\in\mathcal Q(\pi_n)$ such that
\[
\int (y-z)^2\,d Q_n \leq I(\pi_n)+\frac1n.
\]
Since $(Q_n)_{x,z}=\mu\otimes\gamma$ and $\int y^2\,dQ_n=\int y^2\,d\pi_n$, we obtain
\[
\sup_n\int (x^2+y^2+z^2)\,dQ_n<\infty.
\]
Hence, the family $\{Q_n\}$ is uniformly tight. By Prokhorov's theorem, after passing to a subsequence if necessary, we may assume that $Q_n\Rightarrow Q$ for some $Q\in\mathcal P(\mathbb R^3)$. Taking the limits in the marginals, we get
\[
Q_{x,y}=\pi, \qquad Q_{x,z}=\mu\otimes\gamma,
\]
so that $Q\in\mathcal Q(\pi)$. Since the function $(x,y,z)\mapsto (y-z)^2$
is continuous and nonnegative,
\[
\int (y-z)^2\,dQ \leq \liminf_{n\to\infty}\int (y-z)^2\,dQ_n.
\]
Therefore,
\[
I(\pi) \leq \int (y-z)^2\,dQ \leq \liminf_{n\to\infty}\left(I(\pi_n)+\frac1n \right) = \alpha.
\]
Since $\pi\in\mathcal A_c$, the definition of $\alpha$ also gives
$I(\pi)\geq\alpha$. Hence, $I(\pi)=\alpha$.
\end{proof}

\begin{lemma}\label{lem:concave}
The function $V$ is strictly concave on $\mathcal C$.
\end{lemma}

\begin{proof}
We first show that the functional $F(\rho) := W_2^2(\rho, \gamma)$ is strictly convex. Let $\rho_0$ and $\rho_1 \in \mathcal{P}_2(\mathbb R)$. It is well known that
\[
F(t\rho_0 + (1-t) \rho_1) \leq t F(\rho_0) + (1-t) F(\rho_1), \qquad t \in (0, 1),
\]
so it remains to show that equality cannot hold when $\rho_0\ne\rho_1$. Suppose to the contrary that for some $t \in (0, 1)$,
\begin{equation}\label{eq:eq-W}
W_2^2(\rho_t, \gamma) = t W_2^2(\rho_0, \gamma) + (1-t) W_2^2(\rho_1, \gamma),
\end{equation}
where $\rho_t = t \rho_0 + (1-t) \rho_1$. Since $\gamma$ is absolutely continuous, Brenier's theorem yields unique maps $R_k$, $k \in \{0, 1\}$, such that
\[
W_2^2(\rho_k, \gamma) = \inf_{\pi \in \Pi(\rho_k, \gamma)} \int (y-z)^2\, \pi(dy, dz) = \int (R_k(z) - z)^2 \, \gamma(dz).
\]
Thus, the corresponding optimal couplings are $\pi_k = (R_k, \text{id})_\# \gamma$. Set $\pi_t := t \pi_0 + (1-t) \pi_1 \in \Pi\left(\rho_t, \gamma\right)$. By \eqref{eq:eq-W}, this coupling is optimal for $(\rho_t, \gamma)$ and is therefore induced by some map $R_t$. However, by construction, its disintegration with respect to $z$ satisfies
\[
\pi_t^z(dy) = t \delta_{R_0(z)}(dy) + (1-t) \delta_{R_1(z)}(dy) = \delta_{R_t(z)}(dy), \qquad \gamma\text{-a.s.},
\]
which is possible only if $R_0 \overset{\text{\tiny{a.s.}}}{\equiv} R_t \overset{\text{\tiny{a.s.}}}{\equiv} R_1$. Consequently, $\rho_0$ and $\rho_1$ coincide.

We now show that $J(\pi):\mathcal M_2(\mu)\to\mathbb R$ is strictly concave. Let $\pi^0,\pi^1\in\mathcal M_2(\mu)$ and $t\in(0,1)$. Then
\[
(t\pi^0+(1-t)\pi^1)^x=t\pi^{0,x}+(1-t)\pi^{1,x} \qquad \mu\text{-a.s.}
\]
Since the map $\rho\longmapsto W_2^2(\rho,\gamma)$ is strictly convex on $\mathcal P_2(\mathbb R)$, on the set $\{x:\;\pi^{0,x} \ne \pi^{1,x}\}$ of positive $\mu$-measure we have
\[
W_2^2\bigl((t\pi^0+(1-t)\pi^1)^x,\gamma\bigr) < t W_2^2(\pi^{0,x},\gamma)+(1-t) W_2^2(\pi^{1,x},\gamma).
\]
Integrating with respect to $\mu$ and using Lemma~\ref{lem:square}, we obtain the strict concavity of $J$:
\[
J(t\pi^0+(1-t)\pi^1) > t J(\pi^0)+(1-t) J(\pi^1).
\]

Let $c^0, c^1\in\mathcal C$ be distinct call-moment vectors and let $t \in (0,1)$. By Lemma~\ref{lem:optimal-pi-existence}, we may choose $\pi^0,\pi^1\in\mathcal M_2(\mu)$ such that
\[
  \overline C(\pi^k)=c^k, \qquad J(\pi^k) = V(c^k),\qquad k=0,1.
\]
Then $t\pi^0+(1-t)\pi^1\in\mathcal M_2(\mu)$, and
\[
  \overline C(t\pi^0+(1-t)\pi^1) = t c^0 + (1-t) c^1.
\]
By the definition of $V(c)$ and strict concavity of $J$,
\[
V \big(tc^0+(1-t)c^1 \big) \geq J\big(t\pi^0+(1-t)\pi^1\big) > tJ(\pi^0)+(1-t)J(\pi^1) = tV(c^0)+(1-t)V(c^1).
\]
\end{proof}

\begin{remark}
    The proof also shows that $\mathcal M_2(\mu)$ is convex. Since $\pi \longmapsto \overline C(\pi)$ is linear, the set $\mathcal{C}$ is convex as well. Moreover, the strict concavity of $J$ implies the uniqueness of the optimal coupling in problem \eqref{eq:main-problem} for each $c \in \mathcal C$.
\end{remark}

\begin{lemma}\label{lem:penalized}
For every $\lambda\in\mathbb R^N$,
\begin{equation}\label{eq:penalized-value}
\sup_{\pi\in\mathcal M_2(\mu)}\Big\{J(\pi)-\lambda\cdot \overline C(\pi)\Big\} = -\int \Psi_\lambda^*(x)\,\mu(dx).
\end{equation}
\end{lemma}

\begin{proof}
Fix $\lambda\in\mathbb R^N$.
By Lemma \ref{lem:Psi-properties}, $\Psi_\lambda'$ is continuous and non-decreasing. Surjectivity of $\Psi_\lambda'$ follows from \eqref{eq:Psi-prime-growth}. We may therefore define a Borel function
\[
a_\lambda(x):=\inf\{a\in\mathbb R:\ \Psi_\lambda'(a)\geq x\},\qquad x\in\mathbb R.
\]
By continuity of $\Psi_\lambda'$,
\begin{equation}\label{eq:right-inverse}
\Psi_\lambda'(a_\lambda(x))=x,\qquad x\in\mathbb R.
\end{equation}
Thus, $a_\lambda$ is a generalized inverse of $\Psi_\lambda'$. In particular, for each $x\in\mathbb R$, $a_\lambda(x)$ maximizes
\[
a\longmapsto ax - \Psi_\lambda(a),
\]
and hence
\begin{equation*}
\Psi_\lambda(a_\lambda(x))-a_\lambda(x)x=-\Psi_\lambda^*(x).
\end{equation*}
Let $\pi\in\mathcal M_2(\mu)$. For $\mu$-a.e. $x$ and any $q \in \Pi(\pi^x,\gamma)$, we have
\[
yz-\frac12 y^2-\Lambda_\lambda(y) = \Bigl(y(z+a_\lambda(x))-\frac12 y^2-\Lambda_\lambda(y)\Bigr) - a_\lambda(x)y \leq H_\lambda(z+a_\lambda(x)) - a_\lambda(x)y.
\]
Integrating with respect to $q$ and using the martingale condition $\int y\,\pi^x(dy)=x$, we obtain
\[
\int \Bigl(yz-\frac12 y^2-\Lambda_\lambda(y)\Bigr)\,q(dy,dz) \leq \Psi_\lambda(a_\lambda(x))-a_\lambda(x)x = -\Psi_\lambda^*(x).
\]
Taking the supremum over $q\in\Pi(\pi^x,\gamma)$ and integrating with respect to $\mu$, we get
\[
J(\pi)-\lambda\cdot \overline C(\pi) \leq -\int \Psi_\lambda^*(x)\,\mu(dx).
\]
Since $\pi\in\mathcal M_2(\mu)$ is arbitrary,
\begin{equation}\label{eq:upper-penalized}
\sup_{\pi\in\mathcal M_2(\mu)}\Big\{J(\pi)-\lambda\cdot \overline C(\pi)\Big\} \leq -\int \Psi_\lambda^*(x)\,\mu(dx).
\end{equation}

We now prove the reverse inequality. Define a Borel function
\[
Y_x(z):=T_\lambda(a_\lambda(x)+z)
\]
for $(x, z) \in \mathbb R^2$, and set
\[
\rho_\lambda^x:=(Y_x)_\#\gamma,
\qquad
q_\lambda^x(dy,dz):=\delta_{Y_x(z)}(dy)\,\gamma(dz).
\]
Then $q_\lambda^x \in \Pi(\rho_\lambda^x,\gamma)$ for every $x \in \mathbb R$. Let $\pi_\lambda(dx,dy):=\mu(dx)\rho_\lambda^x(dy)$. We show that $\pi_\lambda \in \mathcal M_2(\mu)$. By \eqref{eq:Psi-derivative} and \eqref{eq:right-inverse}, for every $x \in \mathbb R$,
\[
\int y\,\rho_\lambda^x(dy) = \int T_\lambda(a_\lambda(x)+z)\,\gamma(dz) = \Psi_\lambda'(a_\lambda(x)) = x.
\]
Thus, the martingale condition holds. 
It remains to verify square integrability. From \eqref{eq:Psi-prime-growth} and \eqref{eq:right-inverse},
\[
|a_\lambda(x)-x|\leq L_\lambda.
\]
Combining this with \eqref{eq:T-growth}, we get
\[
|Y_x(z)| \leq |a_\lambda(x)+z| + L_\lambda \leq |x|+|z|+2L_\lambda.
\]
Since $\mu,\gamma\in\mathcal P_2(\mathbb R)$,
\[
\int y^2\,\pi_\lambda(dx,dy)<\infty.
\]
Hence, $\pi_\lambda\in \mathcal M_2(\mu)$.

Fix $x\in\mathbb R$. Since
\[
Y_x(z)\in \Gamma_\lambda(a_\lambda(x)+z) \qquad \text{for all } z\in\mathbb R,
\]
we have
\[
H_\lambda(a_\lambda(x)+z) = Y_x(z)(a_\lambda(x)+z) -\frac12 Y_x(z)^2-\Lambda_\lambda(Y_x(z)).
\]
Integrating with respect to $\gamma$, we obtain
\begin{align*}
    \Psi_\lambda(a_\lambda(x)) &= a_\lambda(x)\int Y_x(z)\,\gamma(dz) + \int \Bigl(yz-\frac12 y^2 -\Lambda_\lambda(y)\Bigr)\,q_\lambda^x(dy,dz) \\
    &= a_\lambda(x)\, x + \int \Bigl(yz-\frac12 y^2-\Lambda_\lambda(y)\Bigr)\,q_\lambda^x(dy,dz),
\end{align*}
using the martingale condition. Therefore,
\[
\int \Bigl(yz-\frac12 y^2-\Lambda_\lambda(y)\Bigr)\,q_\lambda^x(dy,dz) = \Psi_\lambda(a_\lambda(x))-a_\lambda(x)\,x = -\Psi_\lambda^*(x).
\]
Since $q_\lambda^x\in\Pi(\rho_\lambda^x,\gamma)$, it follows that
\[
\sup_{q\in\Pi(\rho_\lambda^x,\gamma)} \int \Bigl(yz-\frac12 y^2-\Lambda_\lambda(y)\Bigr)\,q(dy,dz) \geq -\Psi_\lambda^*(x).
\]
Integrating with respect to $\mu$ and evaluating the supremum at $\pi_\lambda$ gives
\[
\sup_{\pi\in\mathcal M_2(\mu)}\Big\{J(\pi)-\lambda\cdot \overline C(\pi)\Big\} \geq -\int \Psi_\lambda^*(x)\,\mu(dx).
\]
Together with \eqref{eq:upper-penalized}, this proves \eqref{eq:penalized-value}.
\end{proof}

\begin{proof}[Proof of Theorem~\ref{thm:main}]
Fix $\lambda\in\mathbb R^N$ and $\pi\in\mathcal M_2(\mu)$ such that $\overline C(\pi)=c$. Then
\[
J(\pi)=J(\pi)-\lambda\cdot \overline C(\pi)+\lambda\cdot c \leq \sup_{\tilde\pi\in\mathcal M_2(\mu)}\{J(\tilde\pi)-\lambda\cdot \overline C(\tilde\pi)\}+\lambda\cdot c.
\]
By Lemma~\ref{lem:penalized},
\[
J(\pi)\leq \lambda\cdot c-\int \Psi_\lambda^*(x)\,\mu(dx).
\]
Taking the supremum over $\pi$ with $\overline C(\pi)=c$ gives
\[
V(c)\leq \lambda\cdot c-\int \Psi_\lambda^*(x)\,\mu(dx).
\]
Taking the infimum over $\lambda\in\mathbb R^N$ proves \eqref{eq:weak-dual}.

Suppose now that $c \in \operatorname{int}(\mathcal C)$. By Lemma~\ref{lem:concave}, $V$ is concave on convex $\mathcal C$, and by Lemma~\ref{lem:square}, it is finite and bounded above by $1/2$. Extend $-V$ to $\mathbb R^N$ by setting
\[
\widetilde V(m):=
\begin{cases}
-V(m), & m\in\mathcal C,\\
+\infty, & m\notin\mathcal C.
\end{cases}
\]
Since $\operatorname{ri}(\mathcal C) = \operatorname{int}(\mathcal C)$, we have $c \in \operatorname{ri}(\operatorname{dom}\widetilde V)$. Hence, by Rockafellar's subdifferentiability theorem \cite[Theorem 23.4]{Rockafellar1970}, there exists $p \in \partial\widetilde V(c)$. Set $\lambda:=-p$. By definition of the subgradient,
\[
\widetilde V(m)\geq \widetilde V(c)+p\cdot (m-c),\qquad m\in\mathbb R^N.
\]
For $m\in\mathcal C$, this is equivalent to
\begin{equation*}
V(m)\leq V(c)+\lambda\cdot (m-c),\qquad m\in\mathcal C.
\end{equation*}
Therefore,
\[
V(c) \geq \lambda\cdot c+\sup_{m\in\mathcal C}\big(V(m)-\lambda\cdot m\big).
\]
By definition of $V$ and Lemma \ref{lem:penalized},
\[
\sup_{m\in\mathcal C}\big(V(m)-\lambda\cdot m\big) = \sup_{\pi\in\mathcal M_2(\mu)}\big(J(\pi)-\lambda\cdot \overline C(\pi)\big) = -\int \Psi_\lambda^*(x)\,\mu(dx).
\]
Consequently,
\[
V(c)\geq \lambda\cdot c-\int \Psi_\lambda^*(x)\,\mu(dx).
\]
Together with \eqref{eq:weak-dual}, this yields equality and proves \eqref{eq:strong-dual}.
\end{proof}

This proof also shows that the optimal value of the dual problem is attained by some $\lambda^*$, allowing us to explicitly construct an optimal coupling $\pi^*$ for the primal problem.

From now on, we assume that a vector $c \in \operatorname{int}(\mathcal C)$ is fixed.

\begin{corollary}\label{cor:penalized-optimality}
Let $\lambda^*\in\mathbb R^N$ be a dual optimizer in \eqref{eq:strong-dual}. Then the coupling
\[
\pi_{\lambda^*}(dx,dy) = \rho_{\lambda^*}^x(dy)\,\mu(dx) = (Y_x)_\#\gamma(dy)\,\mu(dx)
\]
constructed from $\lambda^*$ is optimal for the primal problem, that is,
\[
V(c)=J(\pi_{\lambda^*}), \qquad \overline C(\pi_{\lambda^*})=c.
\]
\end{corollary}
    
\begin{proof}
By Lemma~\ref{lem:penalized}, for every $\lambda\in\mathbb R^N$ coupling $\pi_\lambda$ is optimal for
\[
\sup_{\pi\in\mathcal M_2(\mu)} \Big\{J(\pi)-\lambda\cdot\overline C(\pi)\Big\}.
\]
Combining this with \eqref{eq:strong-dual} for $\lambda=\lambda^*$ gives
\[
J(\pi_{\lambda^*}) -\lambda^*\cdot\overline C(\pi_{\lambda^*}) = V(c) -\lambda^*\cdot c.
\]
By the definition of $V$, $J(\pi_{\lambda^*}) \leq V( \overline C(\pi_{\lambda^*}))$. On the other hand, since $\overline C(\pi_{\lambda^*}) \in \mathcal C$, the weak duality at $m=\overline C(\pi_{\lambda^*})$ gives
\[
V( \overline C(\pi_{\lambda^*})) - \lambda^* \cdot \overline  C(\pi_{\lambda^*}) \leq - \int \Psi_{\lambda^*}^* d\mu = J(\pi_{\lambda^*}) - \lambda^* \cdot \overline C(\pi_{\lambda^*}),
\]
Therefore, $J(\pi_{\lambda^*}) = V( \overline C(\pi_{\lambda^*}))$, and hence
\begin{equation}\label{eq:V-coincide}
V( \overline C(\pi_{\lambda^*})) - \lambda^* \cdot \overline C(\pi_{\lambda^*}) = V(c) - \lambda^* \cdot c.
\end{equation} 
Since $\lambda^*$ is a dual optimizer, \eqref{eq:strong-dual} and
Lemma~\ref{lem:penalized} yield
\[
V(c) - \lambda^* \cdot c = \sup_{m \in \mathcal C}(V(m) - \lambda^* \cdot m).
\]
Thus,
\[
V(m) \leq V(c) + \lambda^* \cdot(m-c), \qquad \forall\, m \in \mathcal{C},
\]
so $\lambda^*$ is a supergradient of $V$ at $c$. By \eqref{eq:V-coincide}, the same supremum is attained at
$\overline C(\pi_{\lambda^*})$. Hence, $\lambda^*$ is also a supergradient of $V$ at
$\overline C(\pi_{\lambda^*})$. Since $V$ is strictly concave by Lemma~\ref{lem:concave}, it cannot have the same supergradient at two distinct points. Therefore, $c$ coincides with $ \overline C(\pi_{\lambda^*})$. Together with $J(\pi_{\lambda^*}) = V \bigl( \overline C(\pi_{\lambda^*}) \bigr)$, this gives
\[
J(\pi_{\lambda^*})=V(c),
\]
as claimed.
\end{proof}

We now describe the structure of the resulting optimal coupling by explicitly computing $H_\lambda$ and $T_\lambda$. 

\begin{corollary}\label{cor:plan-structure}
Let $\lambda^*$ be a dual optimizer for \eqref{eq:strong-dual}, and choose the corresponding optimal coupling in the form: 
\[
\pi^*(dx,dy)=\mu(dx)\,(T_{\lambda^*}(a_{\lambda^*}(x)+\cdot))_\#\gamma(dy).
\]
Then, for $\mu$-a.e. $x\in\mathbb R$, the conditional measure $\pi^{*,x}$ has the following properties:
    \begin{enumerate}[label={\rm(\arabic*)}]
    \item atoms may occur only at the strike prices $K_1, \ldots, K_N$;
    \item away from the strike prices, it is obtained from $\gamma$ by a piecewise-shift transformation.
\end{enumerate}
\end{corollary}

\begin{proof}
For convenience, set $K_0:=-\infty$ and $K_{N+1}:=+\infty$. Then
\[
H_\lambda(u) = \max_{0\leq j \leq N} \sup_{y\in [K_j,K_{j+1}]} \left\{ uy - \frac12 y^2 - \sum_{i=1}^j \lambda_i(y-K_i) \right\} =: \max_{0\leq j \leq N} H_\lambda^j(u).
\]
For fixed $j$, the unconstrained quadratic function under the supremum attains its maximum at $u-\sum_{i=1}^j\lambda_i$. Hence, its supremum over $[K_j,K_{j+1}]$ is attained either at this point, if it belongs to $[K_j,K_{j+1}]$, or at one of the endpoints. Thus, for $j=1,\ldots,N-1$,
\[
H_\lambda^j(u) = 
    \begin{cases}
    u K_j - \frac12 K_j^2 - \displaystyle\sum_{i=1}^j \lambda_i(K_j-K_i),
    & u - \displaystyle\sum_{i=1}^j \lambda_i \leq K_j, \\[1.1em]
    
    \dfrac12 \left( u - \displaystyle\sum_{i=1}^j \lambda_i \right)^2 + \displaystyle\sum_{i=1}^j \lambda_i K_i,
    & K_j \leq u - \displaystyle\sum_{i=1}^j \lambda_i \leq K_{j+1}, \\[1.1em]
    
    u K_{j+1} - \frac12 K_{j+1}^2 - \displaystyle\sum_{i=1}^j \lambda_i (K_{j+1}-K_i), 
    & u - \displaystyle\sum_{i=1}^j \lambda_i \geq K_{j+1}.
    \end{cases}
\]
For $j\in\{0,N\}$, the corresponding formulas reduce to
\[
H_\lambda^0(u) = 
    \begin{cases}  
    \dfrac12 u^2,
    & u \leq K_{1}, \\[1.1em]
    
    u K_{1} - \frac12 K_{1}^2, 
    & \text{otherwise};
    \end{cases}
    \quad
H_\lambda^N(u) = 
    \begin{cases}
    \dfrac12 \left( u - \displaystyle\sum_{i=1}^N \lambda_i \right)^2 + \displaystyle\sum_{i=1}^N \lambda_i K_i,
    \;\  u \geq K_N + \displaystyle\sum_{i=1}^N \lambda_i, \\[1.1em]
    
    u K_N - \frac12 K_N^2 - \displaystyle\sum_{i=1}^N \lambda_i(K_N-K_i),\;
     \text{otherwise}.
    \end{cases}
\]
Thus, $H_\lambda$ is the maximum over $j=0,\ldots,N$ of the quadratic and affine functions displayed above. Hence, $H_\lambda$ is convex and has only finitely many non-differentiability points, while each smooth piece is either quadratic or affine. By Lemma~\ref{lem:H-properties},
\[
T_{\lambda^*}(u) = \min \operatorname*{argmax}_{y \in \mathbb R} \{uy - \frac12 y^2 - \Lambda_{\lambda^*}(y)\}.
\]
Therefore, the above representation of $H_\lambda$ yields
\[
T_{\lambda^*}(u) \in \{K_1, \ldots, K_N\} \cup \left\{ u - \sum_{i=1}^j \lambda^*_i, \, j = 0, \ldots, N \right\}.
\]
It follows that atoms of $\pi^{*,x}$ can occur only at strike prices $K_1,\ldots,K_N$. Away from the strike prices, $\pi^{*,x}$ is obtained from the atomless reference measure $\gamma$ by piecewise shifts of the form $a_{\lambda^*}(x)-\sum_{i=1}^j\lambda_i^*$.
\end{proof}

The following example shows that optimal dual multipliers in the problem \eqref{eq:strong-dual} may be negative.

\begin{example}\label{ex:neg-lambda}
Let $\mu=\delta_0$, $\gamma = \mathcal N(0,1)$, $N=1$, and $K_1=0$. We show that for certain values of $c$, a solution of the dual problem may be negative. Set $\lambda=-1$. Then $\Lambda_\lambda(y)=-(y)_+$ and
\[
T_{-1}(u)=H'_{-1,-}(u)=
    \begin{cases}
        u, & u\leq -\frac12,\\
        u+1, & u>-\frac12.
    \end{cases}
\]
For $a=-\frac12$, we have $\Psi'_{-1}(a)=0$. Indeed, if $Z\sim\gamma$, then
\[
Y(Z):=T_{-1}\left(Z-\frac12\right) =
    \begin{cases}
        Z-\frac12, & Z\leq 0,\\
        Z+\frac12, & Z>0,
    \end{cases}
\]
and therefore $\mathbb E Y(Z)=0$. Moreover,
\[
\mathbb E Y(Z)_+ = \mathbb E\left[\left(Z+\frac12\right)\mathbf 1_{\{Z>0\}}\right] = \frac1{\sqrt{2\pi}}+\frac14.
\]
Hence, setting  $c = \frac1{\sqrt{2\pi}} + \frac14$, the coupling $\pi_{-1}(dx, dy) = \delta_0(dx) Y_\# \gamma(dy)$ satisfies the moment constraint.

We now verify that $\pi_{-1}$ is optimal. Since $\overline C(\pi_{-1})=c$, for any $\pi\in\mathcal M_2(\delta_0)$ with $\overline C(\pi)=c$, Lemma~\ref{lem:penalized} gives
\[
J(\pi) = J(\pi) - \lambda\, \overline C(\pi) + \lambda c \leq J(\pi_{-1}) - \lambda\, \overline C(\pi_{-1}) + \lambda c = J(\pi_{-1}).
\]
Thus, $\pi_{-1}$ is optimal for the primal problem. Moreover, by Lemma~\ref{lem:penalized},
\[
-\int \Psi^*_{-1}\,d\mu = J(\pi_{-1})+\overline C(\pi_{-1}) = V(c)-\lambda c.
\]
Therefore, the dual objective at $\lambda=-1$ equals $V(c)$, so $\lambda=-1$ is a dual optimizer. Notice also that $c$ satisfies the interiority criterion of Proposition~\ref{prop:c-check}.
\end{example}

\section{Connection to  Bass martingales and the MBB problem}\label{sec:mbb}

A fundamental question in the dynamic martingale problem is the existence and uniqueness of a Bass martingale. By \cite[Theorem 1.3 and Remark D.3]{BVBST}, a pair of marginals $(\mu,\nu)$ admits a unique Bass martingale if and only if it is irreducible. We use the following characterization of irreducibility in terms of potentials \cite[Section 3.1]{BackhoffVeraguasBeiglbockHuesmannKallblad2020}.

\begin{definition}
    A pair of measures $(\mu, \nu)$ is called \textit{irreducible} if
    \[
    \{t \in \mathbb R: u_\mu(t) < u_{\nu}(t)\} = \mathbb R, 
    \]
    where $u_\rho(t) := \int |y-t|\, \rho(dy)$ is the potential of the measure $\rho$.
\end{definition}

\noindent
Throughout this section, we assume that $\gamma$ has full support, i.e. $\operatorname{supp} \gamma = \mathbb R$.

\begin{lemma}\label{lem:irreducibility}
Let $\lambda^*$ be a dual optimizer in \eqref{eq:strong-dual}, and let
\[
\pi^*(dx,dy) = \mu(dx)\, \left (T_{\lambda^*}(a_{\lambda^*}(x)+\cdot)\right)_\#\gamma(dy),
\]
be the corresponding optimal coupling, with second marginal $\nu^*$. Then the pair $(\mu,\nu^*)$ is irreducible.
\end{lemma}

\begin{proof}
Fix $t \in \mathbb R$. By Jensen's inequality and the martingale property,
\[
u_{\nu^*}(t) = \int |y-t|\, \pi^{*,x}(dy) \mu(dx) \geq \int \left | \int y\, \pi^{*, x}(dy) - t \right | \mu(dx) = \int |x-t|
\,\mu(dx).
\]
Equality can hold only if $y\mapsto |y-t|$ is affine on a set of full $\pi^{*,x}$-measure for $\mu$-a.e. $x$. Thus, in order to obtain the strict inequality, we show that for every $t, x \in \mathbb R$,
\[
\pi^{*, x}((-\infty, t)) > 0, \qquad \pi^{*, x}((t, \infty)) > 0.
\]
Fix $x \in \mathbb R$. By \eqref{eq:T-growth-maximizers},
\[
a_{\lambda^*}(x) + z -L_{\lambda^*} \leq T_{\lambda^*}(a_{\lambda^*}(x) + z) \leq L_{\lambda^*} + a_{\lambda^*}(x) + z, \qquad z\in\mathbb R.
\]
So, if $z < t-L_{\lambda^*}-a_{\lambda^*}(x)$, then
\[
T_{\lambda^*}(a_{\lambda^*}(x)+z) \leq a_{\lambda^*}(x)+z+L_{\lambda^*} < t.
\]
Since $\gamma$ has full support,
\[
\pi^{*, x}((-\infty, t)) \geq \gamma\left((-\infty,t-L_{\lambda^*}-a_{\lambda^*}(x))\right)>0.
\]
Similarly, if $z>t+L_{\lambda^*}-a_{\lambda^*}(x)$, then
\[
T_{\lambda^*}(a_{\lambda^*}(x)+z) \geq a_{\lambda^*}(x)+z-L_{\lambda^*}>t,
\]
and hence,
\[
\pi^{*, x}((t, \infty)) \geq \gamma\left((t+L_{\lambda^*}-a_{\lambda^*}(x), \infty)\right)>0.
\]
Thus, $u_{\nu^*}(t) > u_\mu(t)$ for every $t\in\mathbb R$, so $(\mu, \nu^*)$ is irreducible.
\end{proof}

We now relate the optimal coupling constructed in Corollary~\ref{cor:plan-structure} to the MBB problem, or equivalently, to the weak martingale transport problem \eqref{eq:wmot}. Throughout the rest of this section, let $\gamma=\mathcal N(0,1)$.

\begin{theorem}
Let $\pi^*$ be the optimal coupling from Corollary~\ref{cor:plan-structure}, and let $\lambda^*$ be the corresponding dual optimizer in
\eqref{eq:strong-dual}. Define
\[
f_{\lambda^*}(y) := \frac{1}{2} y^2 +  \Lambda_{\lambda^*}(y) \in L^1(\nu^*),
\]
where $\nu^*$ is the second marginal of $\pi^*$. Then:
\begin{enumerate}
    \item $\pi^*$ is the law of $(M_0,M_1)$ for the Bass martingale 
    \[
    M_t = \mathbb E\!\left[v'(B_1)\mid B_t\right], \qquad 0\leq t\leq 1,
    \]
    where $v = H_{\lambda^*}$, $v':= v'_- = T_{\lambda^*}$, $B_t:=\xi+W_t$, $\xi \sim (a_{\lambda^*})_\# \mu$, and $(W_t)_{0\leq t\leq 1}$ is a standard Brownian motion independent of $\xi$;
    
    \item $\pi^*$ solves the weak martingale transport problem \eqref{eq:wmot} for its marginals $(\mu, \nu^*)$;
    
    \item The convex envelope $f^{**}_{\lambda^*}$ of $f_{\lambda^*}$ solves the dual problem of \cite{BVBST} for the pair of marginals $(\mu, \nu^*)$:
    \begin{equation}\label{eq:BVBST-dual}
    \inf_{\substack{
    \psi\in L^1(\nu^*),\\ \psi \text{ convex}}} \left \{\int \psi\, d\nu^* - \int (\psi^* * \gamma)^*(x) d\mu \right \}.
    \end{equation}
    Thus, $\lambda^*$ determines a solution of the dual problem with prescribed marginals.
\end{enumerate}
\end{theorem}

\begin{proof}
Since $B_0=\xi$ and $W_1$ is independent of $\xi$ with law $\gamma$,
\[
M_0 = \mathbb E\!\left[v'(\xi+W_1)\mid\xi\right] = \int v'(\xi+z)\,\gamma(dz) = \Psi'_{\lambda^*}(\xi).
\]
Since $\xi\sim(a_{\lambda^*})_\#\mu$ and $\Psi'_{\lambda^*}(a_{\lambda^*}(x))=x$ for all $x$, it follows that $M_0\sim\mu$. Thus, $(M_0,M_1) = (\Psi'_{\lambda^*}(\xi), v'(\xi + W_1))$. Its conditional law given $M_0=x$ is therefore
\[
\bigl(v'(a_{\lambda^*}(x)+\cdot)\bigr)_\#\gamma
=
\rho_{\lambda^*}^x.
\]
Hence, $(M_0,M_1)\sim\pi^*$.

By Lemma~\ref{lem:irreducibility}, the pair $(\mu,\nu^*)$ is irreducible.
Since a Bass martingale between these marginals has just been constructed,
the second assertion follows from \cite[Theorem 1.3]{BVBST}.

To prove the last assertion, we first show that
\[
f^{**}_{\lambda^*} = f_{\lambda^*} \qquad \nu^*\text{-a.e.}
\]
By properties of the Legendre transform, $(f_{\lambda^*}^{**})^* = f_{\lambda^*}^* = H_{\lambda^*}$. Since $f_{\lambda^*}^{**} \leq f_{\lambda^*}$, and for every $u \in \mathbb R$,  $T_{\lambda^*}(u) \in \Gamma_{\lambda^*}(u)$, we have
\[
H_{\lambda^*}(u) = u\, T_{\lambda^*}(u) - f_{\lambda^*}(T_{\lambda^*}(u)) \leq u\, T_{\lambda^*}(u) - f_{\lambda^*}^{**}(T_{\lambda^*}(u)) \leq \left(f_{\lambda^*}^{**}\right)^* (u) = H_{\lambda^*}(u).
\]
Hence, $f_{\lambda^*}^{**}(T_{\lambda^*}(u)) = f_{\lambda^*}(T_{\lambda^*}(u))$. Taking $u=\xi+W_1$ and using the representation of $M_1$ above, we conclude that $f_{\lambda^*}$ coincides with its convex envelope $\nu^*$-a.e.

Since $\nu^*$ satisfies the moment constraints and
$f_{\lambda^*}^{**}=f_{\lambda^*}$ $\nu^*$-a.e.,
\[
\int f_{\lambda^*}^{**}\,d\nu^* = \int f_{\lambda^*}\,d\nu^* = \frac12\int y^2\,\nu^*(dy)+\lambda^*\cdot c.
\]
Moreover, $(f_{\lambda^*}^{**})^* = H_{\lambda^*}$. Therefore, the value of the dual objective in \eqref{eq:BVBST-dual} at
$\psi=f_{\lambda^*}^{**}\in L^1(\nu^*)$ is
\[
\lambda^*\cdot c + \frac{1}{2} \int y^2 \,d\nu^* - \int\Psi_{\lambda^*}^*\,d\mu = V(c) + \frac{1}{2} \int y^2 \,d\nu^*.
\]
Since $J(\pi^*)=V(c)$, and $\pi^*$ solves \eqref{eq:wmot},
\[
V(c) + \frac{1}{2} \int y^2 \,d\nu^* = \int \operatorname{MC}(\pi^{*,x}, \gamma)\, \mu(dx) = \max_{\pi \in \Pi_\mathcal M(\mu, \nu^*)} \int \operatorname{MC}(\pi^{x}, \gamma)\, \mu(dx).
\]
By \cite[Theorem 2.3]{BackhoffVeraguasBeiglbockHuesmannKallblad2020} and \cite[Theorem 1.4]{BVBST}, there is no duality gap for the martingale transport problem with prescribed marginals $\mu$ and $\nu^*$. Hence, $f_{\lambda^*}^{**}$ is a dual optimizer in \eqref{eq:BVBST-dual}.
\end{proof}

\begin{example}\label{ex:concave-equiv}
In Example \ref{ex:neg-lambda}, for $\lambda^*=-1$ and $K_1=0$, we have $ f_{\lambda^*}(y)=\frac12y^2-y_+$, which is not convex. Its convex envelope is obtained by replacing the graph on the interval $\left[-\frac12,\frac12\right]$ by the common tangent to the two quadratic branches at $-\frac12$ and $\frac12$. Therefore,
\[
f_{\lambda^*}^{**}(y)=f_{\lambda^*}(y), \qquad y\notin\left(-\frac12,\frac12\right).
\]
On the other hand, the explicit form of the optimal random variable $Y$ from Example~\ref{ex:neg-lambda} gives
\[
\nu^*\left(\left(-\frac12,\frac12\right)\right)=0.
\]
Consequently, although $f_{\lambda^*}$ is not convex, its convex envelope $f_{\lambda^*}^{**}$ is a dual optimizer in \eqref{eq:BVBST-dual}, and passing to the convex envelope does not change the value of the dual objective.
\end{example}

\printbibliography

\end{document}